\documentclass[10pt,twocolumn,letter]{IEEEtran}
\usepackage{amssymb}
\usepackage{amsmath}
\usepackage{cite}
\usepackage{graphicx}
\usepackage{epstopdf}
\usepackage{amsthm}

\begin{document}

\title{Eigenvectors of Deformed Wigner Random Matrices}
\author{Farzan ~Haddadi and Arash Amini

\thanks{F.H. is with the School of Electrical Engineering, Iran University
of Science \& Technology, Tehran, Iran (e-mail:
farzanhaddadi@iust.ac.ir). A.A. is with the Electrical Engineering
Department of Sharif University of Technology, Tehran, Iran
(e-mail: aamini@sharif.edu).}}

\maketitle

\theoremstyle{plain}
\newtheorem{theorem}{Theorem}
\newtheorem{lemma}[theorem]{Lemma}
\theoremstyle{remark}
\newtheorem{remark}[theorem]{Remark}

\begin{abstract}
We investigate eigenvectors of rank-one deformations of random
matrices $\boldsymbol B = \boldsymbol A + \theta \boldsymbol
{uu}^*$ in which $\boldsymbol A \in \mathbb R^{N \times N}$ is a
Wigner real symmetric random matrix, $\theta \in \mathbb R^+$, and
$\boldsymbol u$ is uniformly distributed on the unit sphere. It is
well known that for $\theta > 1$ the eigenvector associated with
the largest eigenvalue of $\boldsymbol B$ closely estimates
$\boldsymbol u$ asymptotically, while for $\theta < 1$ the
eigenvectors of $\boldsymbol B$ are uninformative about
$\boldsymbol u$. We examine $\mathcal O(\frac{1}{N})$ correlation
of eigenvectors with $\boldsymbol u$ before phase transition and
show that eigenvectors with larger eigenvalue exhibit stronger
alignment with deforming vector through an explicit inverse law.
This distribution function will be shown to be the ordinary
generating function of Chebyshev polynomials of second kind. These
polynomials form an orthogonal set with respect to the semicircle
weighting function. This law is an increasing function in the
support of semicircle law for eigenvalues $(-2\: ,+2)$. Therefore,
most of energy of the unknown deforming vector is concentrated in
a $cN$-dimensional ($c<1$) known subspace of $\boldsymbol B$. We
use a combinatorial approach to prove the result.
\end{abstract}

\begin{IEEEkeywords}
Random matrix, Wigner matrix, eigenvector, rank-one deformation,
phase transition, Catalan number, Chebychev polynomial.
\end{IEEEkeywords}

\section{Introduction}

\IEEEPARstart{L}{et} $\boldsymbol A_{N \! \times \! N}$ be a
random matrix deformed by a low-rank matrix $\boldsymbol P$ to
give $\boldsymbol B = \boldsymbol A + \boldsymbol P$. In this
scenario, $\boldsymbol B$ can be interpreted as the observations
of a structured pure signal $\boldsymbol P$ contaminated by
maximally unstructured noise term $\boldsymbol A$. The main
question is whether reliable information about the signal can be
extracted from noisy observations? We are usually interested in
either an inference on presence of the signal or an estimate of
the signal component \cite{Hogg}. Inference problem on the
presence of an unknown signal entails examining the eigenvalues of
the observation matrix $\boldsymbol B$, specially the largest of
them in magnitude. Therefore, much effort is devoted to
investigating distribution and behavior of eigenvalues of random
matrices. This has been done both in the null hypothesis of a
single random matrix \cite{Wigner, Marcenko, Girko, Tracy}, and in
the alternative hypothesis of a deformed random matrix
\cite{Feral, Benaych}. In contrast, the estimation problem
involves the eigenvectors associated with the largest eigenvalues
of the observation matrix.

$\boldsymbol A_{N \! \times \! N}$ is called a Wigner random
matrix if it is symmetric real with elements $A_{ij}$ independent
random variables for $i\leqslant j$ with zero mean, $\mathbb E
A_{ij}^2 = \frac{1}{N}$, and uniformly bounded higher moments
\cite{Anderson}. Wigner \cite{Wigner}, showed that the eigenvalues
of such a random matrix converge to a bulk with semi-circle law on
the support of $(-2\: ,+2)$. Marcenko and Pastur followed a
similar approach in \cite{Marcenko}, to calculate the distribution
of singular values of a rectangular random matrix. When the
symmetry assumption is relaxed, the complex eigenvalues exhibit a
circular distribution which was observed and sketch-proved by
Girko in \cite{Girko}. Statistical distribution of the largest
eigenvalue of $\boldsymbol A$ is of high importance in inference
and other applications. This distribution was characterized by
Tracy and Widom in \cite{Tracy}.

In an inference scenario, a signal part may be present in the
observations. Signal is a highly structured matrix in the form of
a rank-one unit Ferobenius norm matrix. The observation model will
be  $\boldsymbol B = \boldsymbol A + \theta \boldsymbol {uu}^*$ in
which $\theta \geqslant 0$ is the signal amplitude and $\lVert
\boldsymbol u \rVert _2=1$. It is well-known that if $\theta < 1$,
addition of the signal makes no asymptotic change in the limiting
distribution of the eigenvalues. In case $\theta > 1$, a phase
transition occurs meaning that the largest eigenvalue separates
significantly from the bulk of the spectrum and moves from $+2$ to
essentially $\theta + \frac{1} {\theta}$. This has been shown for
the first time in the context of nonzero mean Wigner matrices in
\cite{Furedi}, then in Gaussian ensembles in \cite{Peche}, and
finally as a universal result with relaxed assumptions on the
random matrix in \cite{Feral}.

The estimation problem is associated with the eigenvectors of the
observation matrix. In the null case when signal is not present,
the eigenvectors of Gaussian random matrix are Haar distributed on
the orthogonal group $\mathbb O(N)$ \cite{Anderson}. When unitary
invariance of a Gaussian distribution is not present, a similar
result \cite{Erdos}, shows that the eigenvectors of a Wigner
random matrix are delocalized in the sense that their $\ell_p$
norms for $p \geqslant 2$ are $\mathcal O(N^{ \frac{1}{p} -
\frac{1}{2}})$. In the deformed case of $\theta > 0$, it is shown
in \cite{Lee} that above a certain threshold for signal
eigenvalues, the observation matrix eigenvectors are partially
localized in the coordinate system defined by signal eigenvectors.
Assumptions on the signal in \cite{Lee} is rather restrictive. A
more general approach in this area is \cite{Benaych} in which
authors demonstrate phase transition both for eigenvalues and
eigenvectors of deformed general random matrix.

For a fixed rank signal, there is a threshold on the eigenvalues
of the signal, above which the corresponding observation matrix
eigenvalue moves out of the bulk in a position predictable by
$\theta$ and Stieltjes transform of the bulk distribution.
Eigenvectors are shown to possess a good alignment with the signal
after phase transition. Assume a rank-one model $\boldsymbol B =
\boldsymbol A + \theta \boldsymbol {uu}^*$ and denote eigenvalue
decomposition of $\boldsymbol A$ and $\boldsymbol B$ as:
\begin{equation}
\label{Aevd} \boldsymbol A = \sum_{i=1}^N  d_i \boldsymbol
\omega_i \boldsymbol \omega_i^*
\end{equation}
\begin{equation}
\label{Bevd} \boldsymbol B = \sum_{i=1}^N  \lambda_i \boldsymbol
v_i \boldsymbol v_i^*
\end{equation}
in which eigenvalues are sorted $d_1 \geqslant d_2 \geqslant
\cdots \geqslant d_N$ and $\lambda_1 \geqslant \lambda_2 \geqslant
\cdots \geqslant \lambda_N$. Assume that $\boldsymbol A$ is a
normalized Wigner real random matrix. Then for $\theta > 1$ the
largest eigenvalue of observation converges to $\lambda_1
\overset{\text{a.s.}} {\longrightarrow} \theta + \frac{1}{\theta}$
\cite{Benaych, Feral} and the associated eigenvector lies
asymptotically on a cone around $\boldsymbol u$ defined by $
\langle \boldsymbol v_1 , \boldsymbol u \rangle  ^2
\xrightarrow{\text{a.s.}} 1 - \frac{1}{\theta^2}$ \cite{Benaych}.
Other eigenvectors are uninformative about $\boldsymbol u$ and
therefore $ \langle \boldsymbol v_i , \boldsymbol u \rangle  ^2
\xrightarrow{\text {a.s.}} 0$.

Before phase transition when $\theta < 1$ every eigenvalue is in
the bulk and the eigenvectors are Haar-distributed on $\mathbb
O(N)$ and therefore $ \langle \boldsymbol v_i , \boldsymbol u
\rangle ^2 \xrightarrow{\text{a.s.}} 0 \; : \; \forall i \leqslant
N$ \cite{Benaych}. In fact, this inner product should sum to one
and therefore it is $\mathcal O \big( \frac{1}{N} \big)$. From a
perturbation perspective, adding the signal part increases the
``energy" of the random matrix in direction $\boldsymbol u$.
Therefore, eigenvectors associated with the largest eigenvalues
should slightly rotate to interpolate between $\boldsymbol A$
powerful directions and $\boldsymbol u$. This seems to result in a
non-uniform distribution of $\boldsymbol u$ energy in subspaces
spanned by each $\boldsymbol v_i$, i.e. $\mathbb E \, | \langle
\boldsymbol v_i , \boldsymbol u \rangle |^2$. Gradually, the first
eigenvectors incorporate a good portion of the energy of
$\boldsymbol u$ and the remaining eigenvectors compete for less.
Therefore most of the energy of $\boldsymbol u$ should be confined
in a subspace spanned by eigenvectors with larger eigenvalues. In
another view, adding energy in direction $\boldsymbol u$,
increases the chance of nearby directions to win to be the
eigenvectors of the largest eigenvalues. Therefore, larger
eigenvalues exhibit better alignment with the signal on average.
Fig. \ref{fig:sample} shows a sample of inner products when $N =
200$ and $\theta = 0.7$.

\begin{figure}
\centering
\includegraphics[width = 0.45 \textwidth]{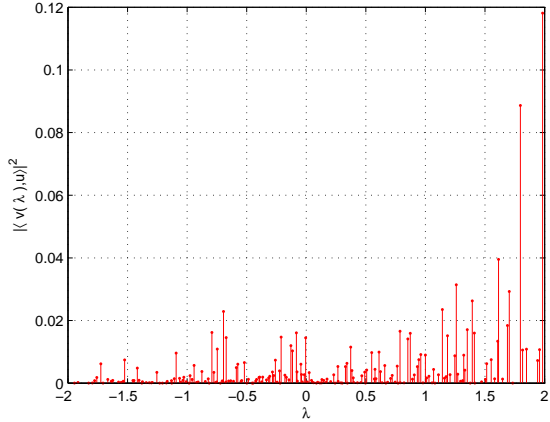}
\caption{A sample of inner products $ \langle \boldsymbol
v(\lambda) , \boldsymbol u \rangle ^2$ versus $\lambda$ in which
$N=200$ and $\theta=0.7$. Eigenvectors associated with larger
eigenvalues are  on average better aligned with $\boldsymbol u$.}
\label{fig:sample}
\end{figure}

\begin{figure}
\centering
\includegraphics[width = 0.45 \textwidth]{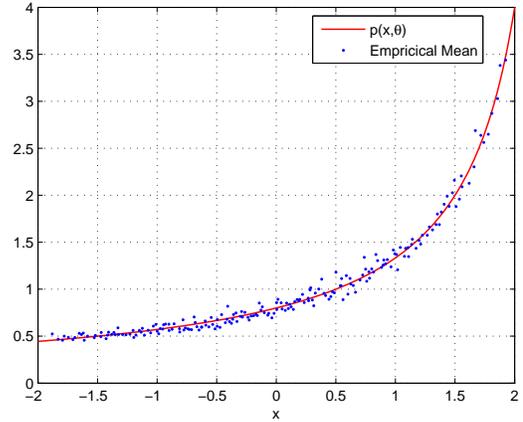}
\caption{Empirical mean of $N  \langle \boldsymbol v(x) ,
\boldsymbol u \rangle ^2$ and the predicted function $p(x;\theta)$
match well. Simulation parameters are $N=200$ and $\theta=0.5$
while 500 Monte Carlo iterations are used to calculate the
empirical mean. } \label{fig:p(x,theta)}
\end{figure}

Nothing is deterministic in Fig. \ref{fig:sample} and therefore we
are interested in the expected value of inner products. Fig.
\ref{fig:p(x,theta)} shows the empirical means of $ \langle
\boldsymbol v(x) , \boldsymbol u \rangle ^2$ in 500 Monte Carlo
iterations with $N=200$ and $\theta=0.5$. In this paper, the law
of distribution $\lim_{N\to+\infty} N \, \mathbb E \, \langle
\boldsymbol v(x),\boldsymbol u \rangle ^2$ is calculated to be:
\begin{equation}
\label{p(x,theta)} p(x;\theta) \mathrel{\mathop:}=
\frac{1}{\theta(\theta + \frac{1}{\theta} -x)}
\end{equation}
which quite matches with the empirical mean in Fig.
\ref{fig:p(x,theta)}.

State of the art signal estimation methods are capable only when
signal is stronger than noise and phase transition has occurred.
Their presupposition is that before phase transition there is no
data extractable about the signal which is lost below noise level.
Though, \eqref{p(x,theta)} shows that this is not the case. Fig.
\ref{fig:p(x,theta)} shows that most of the energy of the unknown
signal is concentrated in the subspace spanned by first
eigenvectors of the observation matrix which are known to the
observer. For example, picking 100 first eigenvectors of Fig.
\ref{fig:p(x,theta)} will give 70$\%$ of the energy of
$\boldsymbol u$. It means that in a 200-dimensional space, before
phase transition, we can specify a 100-dimensional subspace where
the signal is mostly lie in it, which is a lot of information. In
general, regarding that $p(x,\theta)$ is an increasing function
with $x$ when $\theta \geqslant 0$ and for a positive constant $c
\leqslant 1$ we have:
\begin{equation}
\label{pint} \sum_{i=1}^{cN} \langle \boldsymbol v_i , \boldsymbol
u \rangle ^2 \xrightarrow{\text{\; p \;}} \mathcal P (\theta;c)
\geqslant c
\end{equation}
in which $\mathcal P(\theta;C) \leqslant 1$ is an increasing
function of $\theta \geqslant 0$.

After all, the main interesting point about $p(x;\theta)$ is its
extreme simplicity in form. It only has a pole in the location
which it should have. Nothing extra is present in this law. Also
its similarity to the Stieltjes transform kernel seems to be
inherent.

The paper is organized as follows: In section \ref{sec:prior} the
most relevant available results to the problem is discussed.
Section \ref{sec:main} introduces our main contributions and the
proofs are relegated to the appendices.

\section{Prior Art}
\label{sec:prior} In this section, we present the most relative
results in the literature to our main results. These include
results on eigenvalues and eigenvectors of general and random
matrices. We study the spectral decomposition of a real Wigner
matrix perturbed by a rank-one deformation matrix:

\begin{equation}
\label{model} \boldsymbol B_N = \boldsymbol A_N + \theta
\boldsymbol u_N \boldsymbol u_N^*
\end{equation}
where $\boldsymbol u_N$ is a $N \times 1$ vector uniformly
distributed on the unit sphere $\mathbb S^{N-1}$, $\theta \in
\mathbb R^+$ is independent of $N$, and $\boldsymbol A_N$ is a $N
\times N$ real symmetric Wigner matrix defined as:
\begin{equation}
\label{aw} \boldsymbol A_N \mathrel{\mathop :}= \frac{1}{\sqrt
N}\boldsymbol W_N
\end{equation}
in which $\boldsymbol W_N$ is a random matrix with the following
properties:
\\ (i) elements of $\boldsymbol W_N$ are independent up to
symmetry: $ \{W_{ij} \; : \; i \leqslant j\}$ are independent
random variables.
\\ (ii) symmetric distribution and zero odd moments: $\mathbb
E W_{ij}^{2k+1} = 0 \; : \; \forall i,j,k$.
\\ (iii) second moments: $\mathbb E W_{ij}^2 = 1$ for $i<j$ and $\mathbb E
W_{ii}^2$ are uniformly bounded.
\\ (iv) subGaussian assumption: $\forall k \quad \exists \,
\beta > 0 \; : \; \mathbb E W_{ij}^{2\! k} \leqslant (\beta k)^k$.

In this paper, we are mainly concerned with the real setting.
Although, \cite{Feral} assumes an alternative complex setting:
\\(i') diagonal elements are real while $\{W_{ii}\}$ and $\{ \Re
W_{ij} \, , \, \Im W_{ij} \, : \, i<j \}$ are independent real
random variables.
\\(ii') real and imaginary parts are symmetrically distributed
with every odd moments zero.
\\(iii') second moments $\mathbb E |W_{ij}|^2 = 1$ for $i<j$ and $\mathbb E
W_{ii}^2$ are uniformly bounded.
\\ (iv') subGaussian assumption: $\forall k \quad \exists \,
\beta > 0 \; : \; \mathbb E |W_{ij}|^{2\! k} \leqslant (\beta
k)^k$.

Since the deforming matrix $\theta \boldsymbol u_N \boldsymbol
u_N^*$ is of fixed rank (here rank-one), it will not
asymptotically affect the global distribution of the eigenvalues
of $\boldsymbol B$. The distribution is still the original Wigner
semicircle law of the eigenvalues of $\boldsymbol A$. The
empirical distribution of eigenvalues under assumptions (i)-(iv)
or (i')-(iv') converges weakly to the probability measure:
\begin{equation}
\label{weak} \mu_N \mathrel{\mathop:}= \frac{1}{N} \sum_{i=1}^N
\delta_ {d_i} \xrightarrow{ \text{N}\to \infty} \mu_{\textrm{sc}}
\end{equation}
with corresponding density of a semicircle:
\begin{equation}
\label{semicircle} \frac{\mathrm{d} \mu_{\textrm{sc}}
(x)}{\mathrm{d} x} = \frac{1}{2\pi} \sqrt{4-x^2} \, \mathbb
I_{[-2,+2]}(x)
\end{equation}

From hereafter we may omit inherent dependence of variables on $N$
for better readability. Although addition of a fixed rank
perturbation does not alter the global behavior of the
eigenvalues, it may strongly influence the extreme eigenvalues.
The following result was first presented for Gaussian Wigner
random matrices in \cite{Peche}, and then generalized for any
subGaussian ensemble (i')-(iv') in \cite{Feral}:

\begin{theorem}
\label{th:feral} \cite{Feral} For any real $t \geqslant 0$ define:
\begin{equation}
\label{rho} \rho_\theta \mathrel{\mathop:}=
\theta+\frac{1}{\theta}
\end{equation}
\begin{equation}
\label{sigma_t} \sigma_\theta^2 \mathrel{\mathop:}= 1 -
\frac{1}{\theta^2}
\end{equation}
$\bullet$ for $\theta>1$ :
\begin{equation}
\label{l_1_gaus} \lim_{N \to \infty} \mathbb P \{ N^{\frac{1}{2}}
(\lambda_1 - \rho_\theta) \geqslant t \} = \frac{1}{\sqrt{2\pi}
\sigma_\theta} \int_{-\infty}^t \!\!\! e^{- \frac{y^2}{2
\sigma_\theta^2}} \textnormal dy
\end{equation}
$\bullet$ for $\theta < 1$ :
\begin{equation}
\label{l_1_tw2} \lim_{N \to \infty} \mathbb P \{ N^{\frac{2}{3}}
(\lambda_1 - 2) \geqslant t \} = F_2^{\textnormal{TW}}(t)
\end{equation}
$\bullet$ for $\theta = 1$ :
\begin{equation}
\label{l_1_tw2} \lim_{N \to \infty} \mathbb P \{ N^{\frac{2}{3}}
(\lambda_1 - 2) \geqslant t \} = F_3^{\textnormal{TW}}(t)
\end{equation}
where $F_2^{\textnormal{TW}}$ and $F_3^{\textnormal{TW}}$ are
Tracy-Widom distributions with 2 and 3 degrees of freedom
\cite{Tracy}.
\end{theorem}

Basically, Theorem \ref{th:feral} states that before and on the
phase transition, the largest eigenvalue is approximately
unchanged by the presence of the deforming factor, while after
phase transition it is moved out of the bulk of the spectrum to
its new position at $\rho_\theta$.

Despite the eigenvalues, little is known about the eigenvectors of
deformed random matrices. Using the Stieltjes transform
\cite{Bai}, it was shown in \cite{Benaych} that the eigenvectors
also experience a phase transition.

\begin{theorem}
\label{th:benaych} \cite{Benaych} For $\theta \geqslant 0$ and
under assumptions (i)-(iv):
\begin{equation*}
\label{v_1_bef} \langle \boldsymbol v_1 , \boldsymbol u \rangle ^2
\xrightarrow{\text{a.s.}} \left \{
\begin{array}{ll}
1- \frac{1}{\theta^2} & \textnormal{if} \quad \theta > 1 \\
0 & \textnormal{if} \quad \theta \leqslant 1
\end{array} \right.
\end{equation*}
as $N \to +\infty$.
\end{theorem}
Theorem \ref{th:benaych} states the phase transition for
eigenvectors and predicts two distinct phases for their
distribution with respect to the signal component $\boldsymbol u$.
Before phase transition no information is available about
$\boldsymbol u$ in eigenvectors $\boldsymbol v_i$, while after
phase transition a single eigenvector $\boldsymbol v_1$ bears a
large amount of information about $\boldsymbol u$. Another
relevant result about perturbation of eigenvectors of a general
matrix is the Davis-Kahan inequality \cite{Davis}:

\begin{theorem}
\label{th:davis} (Davis-Kahan \cite{Vershynin}) if $\boldsymbol S$
and $\boldsymbol T$ are symmetric $N \times N$ matrices and if:
\begin{equation}
\label{delta} \delta \mathrel{\mathop :}= \min_{j \neq i}
|\lambda_i(\boldsymbol S) - \lambda_j(\boldsymbol S)|
\end{equation}
then the angle between corresponding eigenvectors is bounded above
by:
\begin{equation}
\label{sin} \sin \angle (\boldsymbol v_i(\boldsymbol S) ,
\boldsymbol v_i(\boldsymbol T)) \leqslant \frac{2 \, \|
\boldsymbol S - \boldsymbol T \|}{\delta}
\end{equation}
\end{theorem}

Suppose that in Theorem \ref{th:davis} we set $\boldsymbol S
\mathrel{\mathop :}= \boldsymbol B$ and $\boldsymbol T
\mathrel{\mathop :}= \theta \boldsymbol {uu}^*$. Since
$\|\boldsymbol A \| \simeq 2$ and the level spacing between
eigenvalues of $\boldsymbol B$ is of $\mathcal O(\frac{1}{N})$
before phase transition, Davis-Kahan gives an upper bound of
$\mathcal O(N)$ in \eqref{sin} which is useless.

\section{Main Results}
\label{sec:main} In this section the main results of this paper is
presented while proofs are relegated to the appendices. In the
low-rank deformation problem for random matrices, the distribution
of the first eigenvectors are well known after phase transition
\cite{Benaych}. Though, little is known about the situation before
phase transition. This is because of the premise that eigenvectors
``individually" might carry useful information about the signal
subspace. In this regard, Theorem \ref{th:benaych} shows that this
information is zero asymptotically. But simulation results e.g.
Fig. \ref{fig:sample} exhibit a random structure in correlations
of eigenvectors with the signal $\boldsymbol u$. Although this
information is $\mathcal O(\frac{1}{N})$, it follows a very smooth
increasing expected value which can aggregate information of
$\mathcal O(N)$ first eigenvectors to achieve a meaningful
information about the signal $\boldsymbol u$. Therefore, it is
useful to study this small correlation.

\begin{theorem}
\label{th:p(x,t)} For a rank-one deformation model of
\eqref{model} and under assumptions (i)-(iv) for the Wigner matrix
$\boldsymbol A$:
\begin{equation}
\label{p(x,theta)th} p(x;\theta) \mathrel{\mathop :} = \lim_{N\to
\infty} N \, \mathbb E \, \langle \boldsymbol v(x),\boldsymbol u
\rangle ^2 = \frac{1} {\theta (\theta + \frac{1} {\theta}-x)}
\end{equation}
in which $\boldsymbol v(x)$ is the eigenvector of $\boldsymbol B$
corresponding to eigenvalue $x$.
\end{theorem}

\begin{proof}
See Appendix \ref{app:p(x,t)}.
\end{proof}

\begin{remark}
The distribution function in Theorem \ref{th:p(x,t)} is
surprisingly equal to the ordinary generating function of the
Chebyshev polynomials of second kind $U_k(\frac{x}{2})$. These
polynomials form an orthogonal set of polynomials with an inner
product weighted by the semicircle function.

Function $p(x;\theta)$ has a simple pole at $\theta +
\frac{1}{\theta}$. Therefore, any eigenvalue located around this
value will get a large inner product $\langle \boldsymbol v(\theta
+ \frac{1}{\theta}) , \boldsymbol u \rangle ^2 = \mathcal O(1)$
while other eigenvalues exhibit $\mathcal O(\frac{1}{N})$ inner
products. Before phase transition, every eigenvalue is in the bulk
of semicircle law supported on $(-2,+2)$ while $\theta + \frac{1}
{\theta}$ is outside of this interval. For $\theta=1$, as $N
\rightarrow +\infty$ we have $\lambda_1 \xrightarrow{\text{a.s.}}
2$ \cite{Anderson}, and therefore $\mathbb E \langle \boldsymbol
v_1 , \boldsymbol u \rangle ^2$ will become large. This means that
\eqref{p(x,theta)} predicts the phase transition of $\boldsymbol
v_1$ at $\theta=1$. Now suppose that $\theta > 1$ and we will have
an eigenvalue located around $\theta + \frac{1}{\theta}$
\cite{Benaych}. Then \eqref{p(x,theta)th} predicts that
$\boldsymbol v_1$ is well-aligned with $\boldsymbol u$. Therefore,
\eqref{p(x,theta)th} describes the eigenvectors behavior before,
after, and on the phase transition.
\end{remark}

Theorem \ref{th:p(x,t)} can be used to describe the distribution
of the eigenvectors of the deformed random matrix both before and
after phase transition in a single law. Although the mean value of
correlations are smooth, their samples exhibit a random behavior.
Therefore, to estimate a subspace close to $\boldsymbol u$,
sufficient number of first eigenvectors should be incorporated in
a span. This subspace will contribute a concentrated portion of
energy of $\boldsymbol u$ larger than its proportional dimension:

\begin{theorem}
\label{th:sum} For a single matrix $\boldsymbol B$ abiding model
\eqref{model} and under assumptions (i)-(iv) on $\boldsymbol A$:
\begin{equation}
\label{sum} \sum_{i=1}^{cN} \langle \boldsymbol v_i , \boldsymbol
u \rangle ^2 \xrightarrow{\; \; \textnormal p \;\;} \mathcal P
(\theta;c) \geqslant c
\end{equation}
where $c \leqslant 1$ is a positive constant and the function
$\mathcal P(\theta;c)$ is an increasing function of both $c$ and
$\theta$:
\begin{equation}
\label{p_int} \mathcal P(\theta;c) \mathrel{\mathop :}= \int_m^2
\mu_{\textnormal{sc}}(x) p(x;\theta) \textnormal d x
\end{equation}
in which $m$ is a threshold defined implicitly via:
\begin{equation}
\label{m} c = \int_m^2 \mu_{\textnormal{sc}}(x) \textnormal dx
\end{equation}
\end{theorem}

\begin{proof}
See Appendix \ref{app:sum}.
\end{proof}

\begin{remark}
Theorem \ref{th:sum} paves the way for using Theorem
\ref{th:p(x,t)} in practice using a sum in the spectrum which
concentrates around the expected value. In practice, only one
sample of the deformed matrix is available and therefore, we
cannot use a mean value to approach the expected value. Theorem
\ref{th:sum} gives an alternative way for averaging by
incorporating large number of stronger eigenvectors to achieve a
good estimate of the signal subspace.
\end{remark}

\appendices
\section{Proof of Theorem \ref{th:p(x,t)}}
\label{app:p(x,t)} We are interested in inner products of the
eigenvectors of $\boldsymbol B$ and $\boldsymbol u$. The classical
Wigner proof \cite{Wigner}, for the semicircle distribution of the
eigenvalues of a symmetric random matrix used traces of the random
matrix and its powers. Trace of $k^{\textnormal {th}}$ power of
the matrix corresponds to the $k^{\textnormal {th}}$ moment of its
eigenvalues distribution. We use the same idea to calculate the
distribution of the inner products. To produce such inner products
multiply $\boldsymbol B$ by $\boldsymbol u$ from left and right:
\begin{equation}
\label{ubu} \boldsymbol u^* \boldsymbol {Bu} = \sum_{i=1}^N
\lambda_i \langle \boldsymbol v_i , \boldsymbol u \rangle ^2
\end{equation}
The same can be done for the $k^{\textnormal {th}}$ power of
$\boldsymbol B$:
\begin{equation}
\label{ubku} \boldsymbol u^* \boldsymbol B^k u = \sum_{i=1}^N
\lambda_i^k \langle \boldsymbol v_i , \boldsymbol u \rangle ^2
\end{equation}

These are linear combinations of the inner products. These
quadratic forms have been used in the literature to show
localization properties of eigenvectors of random matrices
\cite{Wang}. Assume that phase transition is not occurred and then
the distribution of $\lambda_i$ is known. Therefore, we are able
to deduce distribution of the inner products from sufficient
different linear combinations in the form of \eqref{ubku}. Using
the model in \eqref{model}, we can calculate the value of linear
combinations e.g. \eqref{ubu} in terms of $\theta$:
\begin{equation}
\label{sum1_t} \boldsymbol u^* \! \boldsymbol {Bu} = \boldsymbol
u^* (\boldsymbol A + \theta \boldsymbol u \boldsymbol u^*)
\boldsymbol u = \boldsymbol u^* \! \boldsymbol {Au} + \theta
\end{equation}
The second equality comes from the fact that $\boldsymbol u^*
\boldsymbol u = 1$. Since $\boldsymbol u$ is uniformly distributed
on the unit sphere $\mathbb S^{N-1}$ and $\boldsymbol A$ is a
subGaussian random matrix, the product form $ \boldsymbol u^*
\boldsymbol {Au}$ is concentrated around its mean. The following
Lemma states the result:

\begin{lemma}
\label{lem:uau} For a Wigner matrix $\boldsymbol A$ with
assumptions (i)-(iv) and $\boldsymbol u$ uniformly distributed on
the unit sphere $\mathbb S^{N-1}$ and $k\in \mathbb N \cup \{0\}$:
\begin{equation}
\label{conc_uau} \boldsymbol u^* \! \boldsymbol A^k u \xrightarrow
{\; \; \textnormal p \; \;} \left \{
\begin{array}{ll}
c_{{k}/{2}}  & k: \; \textnormal {even} \\
0  & k: \; \textnormal {odd}
\end{array} \right.
\end{equation}
as $N \to \infty$ and $c_k$ is the $k^{\textnormal{th}}$ Catalan
number:
\begin{equation*}
\label{catalan} c_k \mathrel{\mathop :} = \frac{1}{k+1} \left(
\!\!\!
\begin{array}{c}
    2k \\ \! k
\end{array} \!\!\! \right)
\end{equation*}
\end{lemma}
\begin{proof}
See Appendix \ref{app:uau}.
\end{proof}

Before phase transition $\lambda_i$'s are in the bulk spectrum
with spacing of $\mathcal O(\frac{1}{N})$. Assuming that the
expected values of the inner products in \eqref{ubku} is a smooth
function of the eigenvalues, we will have the following Lemma:
\begin{lemma}
\label{lem:ubu_int} the quadratic form in \eqref{ubku} converges
in probability to its mean value:
\begin{equation}
\label{ubu_int} \boldsymbol u^* \boldsymbol B^k u \xrightarrow{\;
\; \textnormal p \;\;} \int \mu_{\textnormal {sc}} (x) x^k
p(x;\theta) \textnormal dx
\end{equation}
in which
\begin{equation}
\label{p_lem} p(x;\theta) \mathrel{\mathop :}= N \mathbb E \{
\langle \boldsymbol v(x) , \boldsymbol u \rangle ^2 | x \}
\end{equation}
\end{lemma}
\begin{proof}
See Appendix \ref{app:ubu_int}.
\end{proof}

The distribution function $p(x;\theta)$ is assumed to be a smooth
function of $x$ and $\theta$. Therefore, it has a Taylor series
with respect to $\theta$:
\begin{equation}
\label{p_taylor} p(x;\theta) = \sum_{k=0}^{\infty} \theta^k f_k(x)
\end{equation}
Combinatorial calculations show that $f_k(x)$ are Chebyshev
polynomials of second kind. These polynomials form an orthogonal
polynomial set with respect to the weight function of the
semicircle law $\mu_{\textnormal {sc}}(x)$, in the interval $[-2
\, , +2]$. The first few functions are:
\begin{eqnarray}
\label{fk_ex} f_0(x) & = & 1 \nonumber \\
f_1(x) & = & x \nonumber \\
f_2(x) & = & x^2-1 \nonumber \\
f_3(x) & = & x^3-2x \nonumber \\
f_4(x) & = & x^4-3x^2+1
\end{eqnarray}

\begin{lemma}
\label{lem:chebyshev} Polynomials $f_k(x)$ of the Taylor series
expansion of the distribution function $p(x;\theta)$ are described
as:
\begin{equation}
\label{f_cheb} f_k(x) = U_k \left(\frac{x}{2} \right)
\end{equation}
in which $U_k(x)$ is the Chebyshev polynomial of second kind.
\end{lemma}

\begin{proof}
See Appendix \ref{app:cheb}.
\end{proof}

Therefore, the distribution function $p(x;\theta)$ is the ordinary
generating function of Chebyshev polynomials of second kind which
is known to be \cite{Casarano}:
\begin{equation}
\label{p_gen} p(x;\theta) = \sum_{k=0}^\infty \theta^k U_k \left(
\frac {x}{2} \right) = \frac{1}{1-\theta x + \theta^2}
\end{equation}
for $|\theta| < 1$ which is equivalent to \eqref{p(x,theta)th}.

\section{Proof of Theorem \ref{th:sum}}
\label{app:sum} To show convergence of the summation to its limit
in probability, we first show that the limit is the expected value
of the sum and then investigate the second moment.
\begin{lemma}
Under the assumptions of Theorem \ref{th:sum}:
\begin{equation}
\label{Esum} \mathbb E \sum_{i=1}^{cN} \langle \boldsymbol v_i ,
\boldsymbol u \rangle ^2 = \mathcal P(\theta;c)
\end{equation}
\end{lemma}
\begin{proof}
 The sum in \eqref{Esum} can be converted to:
\begin{equation}
\label{sum_ind} \sum_{i=1}^{cN} \langle \boldsymbol v_i ,
\boldsymbol u \rangle ^2 = \sum_{i=1}^{N} \langle \boldsymbol v_i
, \boldsymbol u \rangle ^2 \, \mathbb I(\lambda_i \geqslant
\lambda_{cN})
\end{equation}
in which $\mathbb I(\cdot)$ is the indicator function. Whatever
$\theta$ is, $\lambda_{cN}$ converges in probability to $m$ which
is defined implicitly by \eqref{m}. Using the techniques of the
proof of Lemma \ref{lem:Eubu} in Appendix \ref{app:ubu_int}, the
expected value of \eqref{sum_ind} will be
\begin{align}
\mathbb E \, \sum_{i=1}^{cN} \langle \boldsymbol v_i , \boldsymbol
u \rangle ^2 & \xrightarrow{\; \; \textnormal p \;\;} \frac{1}{N}
\sum_{i=1}^{N} \, \mathbb E \, \{ \mathbb I(\lambda_i \geqslant m)
\, \mathbb E \, \{ N \langle \boldsymbol v_i , \boldsymbol u
\rangle ^2 | \lambda_i \} \}
\\
& = \quad \frac{1}{N} \sum_{i=1}^{N}  \, \mathbb E \, \{ \mathbb
I(\lambda_i \geqslant m) \, p(\lambda_i ; \theta) \}
\\
\label{sig-int} & \xrightarrow{\; \; \textnormal p \;\;} \int
\mu_{\textnormal {sc}} (x) \, \mathbb I(x \geqslant m) \,
p(x;\theta) \textnormal
dx \\
& = \quad \int_m^2 \mu_{\textnormal {sc}} (x) \, p(x;\theta)
\textnormal dx = \mathcal P (\theta;c)
\end{align}
where in \eqref{sig-int} we have used the fact that
asymptotically, the empirical measure of eigenvalues converges
weakly in probability to the semi-circle law. Note that the above
results are valid only for $c < 1$.
\end{proof}

Lemma \ref{lem:ubu_int} states that $\boldsymbol u^* \boldsymbol
B^k \boldsymbol u$ converges in probability to its mean. Define
the empirical probability measure and its mean as:
\begin{align}
L_N \mathrel{\mathop :} &= \sum_{i=1}^N \delta _{\lambda_i}
\langle \boldsymbol v_i ,
\boldsymbol u \rangle ^2 \\
\langle \bar{L}_N , f \rangle \mathrel{\mathop :} &= \mathbb E \,
\langle L_N , f \rangle  \quad \forall f \in C_b
\end{align}
Then, $\boldsymbol u^* \boldsymbol B^k \boldsymbol u = \langle L_N
, x^k \rangle$ and Lemma \ref{lem:ubu_int} asserts that
\begin{align}
\lim_{N \to \infty} \langle L_N , x^k \rangle = \langle \,
\mu_{\textnormal {sc}}(x) \, p(x;\theta) , x^k \rangle
\end{align}
Although function $\mathbb I(x \geqslant m)$ is not continuous, a
deliberately exact approximation of it can be formed by a
truncated Taylor series expansion in the interval $[-2 , +2]$.
Therefore we can conclude that
\begin{align}
\lim_{N \to \infty} & \sum_{i=1}^{cN} \langle \boldsymbol v_i ,
\boldsymbol u \rangle ^2 = \lim_{N \to \infty} \langle L_N ,
\mathbb I(x \geqslant m) \rangle  = \\
& \langle \, \mu_{\textnormal {sc}}(x) \, p(x;\theta) , \mathbb
I(x \geqslant m) \rangle = \mathcal P (\theta;c)
\end{align}
which is the result of Theorem \ref{th:sum}.

\section{Proof of Lemma \ref{lem:uau}}
\label{app:uau} We show that $\boldsymbol u^* \boldsymbol A^k
\boldsymbol u$ concentrates around its mean:
\begin{eqnarray}
\label{Euau} \mathbb E \, \boldsymbol u^* \boldsymbol A^k
\boldsymbol u = \mathbb E \, \textnormal {Tr} \, (\boldsymbol u^*
\boldsymbol A^k \boldsymbol u) = \mathbb E \, \textnormal {Tr} \,
( \boldsymbol A^k \boldsymbol u \boldsymbol u^* ) = \nonumber \\
\mathbb E_{\boldsymbol A} \, \mathbb E_{\boldsymbol u} \{
\textnormal {Tr} \, (\boldsymbol A^k \boldsymbol {uu}^* ) \, |
\boldsymbol A \} = \mathbb E_{\boldsymbol A} \, \textnormal {Tr}
\, (\boldsymbol A^k \, \mathbb E_{\boldsymbol u} \boldsymbol
{uu}^* ) \nonumber \\
= \mathbb E \, \textnormal {Tr} \, (\boldsymbol A^k
\frac{1}{N}\boldsymbol I_N) =  \mathbb E \, \frac{1}{N}
\textnormal {Tr} \boldsymbol A^k
\end{eqnarray}
which is $c_{k/2}$ when $k$ is even and $0$ otherwise. In
\eqref{Euau} we have used the fact that a random vector uniformly
distributed on the unit sphere $\mathbb S^{N-1}$ is isotropic
\cite{Vershynin} and therefore $\mathbb E \, \boldsymbol {uu}^* =
\frac{1}{N}\boldsymbol I_N$.

Now we give a Gaussian comparison and upper bound for product
moments of $\boldsymbol u$:

\begin{lemma}
\label{lem:comparison} if $\boldsymbol u \sim \textnormal{unif}
(\mathbb S^{N-1})$ and $\boldsymbol x \sim \mathcal N (0 ,
\frac{1}{N} \boldsymbol I_N)$ are two random vectors then:
\begin{equation}
\label{comp} \mathbb E \, u_i^2 u_j^2 \leq \mathbb E \, x_i^2
x_j^2 = \left \{
\begin{array}{lrl}
3/N^2  & \textnormal {if} & i=j \\
1/N^2  & \textnormal {if} & i\neq j
\end{array} \right.
\end{equation}
\end{lemma}

\begin{proof}
Although $\boldsymbol u$ is a sub-Gaussian random vector
\cite{Vershynin}, this is not enough to infer \eqref{comp}.
Moments of sub-Gaussian random vectors are bounded above by
Gaussian moments times a constant while in \eqref{comp} the
constant is unity.

General product moments of uniform distribution on unit sphere is
derived in \cite{Fang}:
\begin{equation}
\label{prod} \mathbb E \prod_{i=1}^N u_i^{k_i} =
\frac{\Gamma(\frac{N}{2})}{2^k \, \Gamma(\frac{N+k}{2})} \,
\prod_{i=1}^{N} \frac{k_i!}{(\frac{k_i}{2})!}
\end{equation}
in which $k = \sum_{i=1}^N k_i$ and every $k_i$ should be an even
number. Therefore, for the four'th moment we will have:
\begin{equation}
\label{4moment} \mathbb E \, u_i^4 =
\frac{\Gamma(\frac{N}{2})}{2^4 \, \Gamma(\frac{N}{2}+2)} \,
\frac{4!}{2!} = \frac{3}{N^2+2N} \leqslant \frac{3}{N^2} = \mathbb
E \, x_i^4
\end{equation}
and the product moment $i \neq j$ will be bounded as:
\begin{equation}
\label{22moment} \mathbb E \, u_i^2 u_j^2 = \frac{1}{N^2+2N}
\leqslant \frac{1}{N^2} = \mathbb E \, x_i^2 x_j^2
\end{equation}
and the lemma is proved.
\end{proof}

Next, we examine the second moment of $\boldsymbol u^* \boldsymbol
A^k \boldsymbol u$ to show concentration around its mean. Define
$\boldsymbol F \mathrel{\mathop :} = \boldsymbol A^k$:
\begin{eqnarray}
\label{mom2} \mathbb E \, (\boldsymbol u^* \! \boldsymbol A^k
\boldsymbol u )^2 = \mathbb E \left( \sum_{i,j} u_i F_{ij} u_j
\right)^2 = \nonumber \\
\mathbb E \sum_{i,j,m,n} u_i u_j u_m u_n F_{ij} F_{mn}
\end{eqnarray}
Although elements of $\boldsymbol u$ are not independent random
variables, it can be shown that the expected value of any
combination of its elements with odd powers is zero due to the
symmetry in the sphere \cite{Vershynin}. The surviving terms are:
\begin{eqnarray}
\mathbb E \bigg( \sum_i u_i^4 F_{ii}^2 + \sum_{i,j \neq i}
u_i^2 u_j^2 F_{ii} F_{jj} + \nonumber \\
\quad \quad \sum_{i,j \neq i} u_i^2 u_j^2 F_{ij}^2 + \sum_{i,j
\neq i} u_i^2 u_j^2 F_{ij}^2 \bigg)
\end{eqnarray}
corresponding to situations where $(i=j=m=n)$, $(i=j , m=n)$ ,
$(i=m , j=n)$, and $(i=n , j=m)$ in \eqref{mom2}. We have also
used the fact that $\boldsymbol F$ is a symmetric matrix. Using
upper bounds in Lemma \ref{lem:comparison} for the expectation on
$\boldsymbol u$ we get:
\begin{eqnarray}
\mathbb E \, (\boldsymbol u^* \! \boldsymbol A^k \boldsymbol u )^2
\quad \quad \quad \quad \quad \quad \quad \quad \quad \quad \quad
\quad \quad \quad \quad \quad \nonumber \\
\leqslant \mathbb E \frac{1}{N^2} \bigg( 3 \sum_i F_{ii}^2 +
\sum_{i \neq j} F_{ii} F_{jj} + 2 \sum_{i \neq j} F_{ij}^2 \bigg)
\nonumber \\
= \mathbb E \frac{1}{N^2} \bigg( \sum_{i,j} F_{ii} F_{jj} + 2
\sum_{i,j} F_{ij}^2 \bigg) \quad \quad \quad \quad \nonumber \\
= \mathbb E \bigg( \frac{1}{N}\textnormal {Tr} \boldsymbol A^{k}
\bigg)^2 + \mathbb E \frac{2}{N^2} \, \textnormal {Tr} \boldsymbol
A^{2k} \quad \quad \quad \quad \nonumber \\
\xrightarrow {\;\; \textnormal p \;\;} \bigg( \mathbb E
\frac{1}{N}\textnormal {Tr} \boldsymbol A^{k} \bigg)^2 +
\frac{2}{N} c_k \quad \quad \quad \quad
\end{eqnarray}
in which, the convergence in probability is due to Wigner in its
proof of semi circle law \cite{Anderson}. Therefore:
\begin{equation}
\textnormal {Var} (\boldsymbol u^* \! \boldsymbol A^k \boldsymbol
u) \leqslant \frac{2}{N} c_k
\end{equation}
Now, a standard Chebychev inequality shows concentration around
the mean value in \eqref{Euau}:
\begin{equation}
\mathbb P (| \boldsymbol u^* \! \boldsymbol A^k \boldsymbol u -
\mathbb E \, \boldsymbol u^* \! \boldsymbol A^k \boldsymbol u |
\geqslant t ) \leqslant \frac{\frac{2}{N} c_k}{t^2} \xrightarrow
{N \to \infty} 0
\end{equation}
for each fixed $t$ and therefore, convergence in probability to
the mean value is proved.

\section{Proof of Lemma \ref{lem:ubu_int}}
\label{app:ubu_int} We will show that $\boldsymbol u^* \boldsymbol
B^k \boldsymbol u$ converges in probability to its expected value.
\begin{lemma}
\label{lem:Eubu} The expected value of the quadratic form is:
\begin{equation}
\label{Eubu} \mathbb E \, \boldsymbol u^* \boldsymbol B^k
\boldsymbol u = \int \mu_{\textnormal {sc}}(x) \, x^k p(x;\theta)
\, \textnormal d x
\end{equation}
in which
\begin{equation}
p(x;\theta) \mathop : = \lim_{N \to \infty}  \mathbb E\, \{ N
\langle \boldsymbol v_i , \boldsymbol u \rangle ^2 | \lambda_i = x
\}
\end{equation}
\end{lemma}
\begin{proof}
Left Hand Side of \eqref{Eubu} can be written as:
\begin{eqnarray}
\mathbb E \, \sum_{i=1}^{N} \langle \boldsymbol v_i , \boldsymbol
u \rangle ^2 \lambda_i^k =  \frac{1}{N} \sum_{i=1}^N \mathbb E \,
\{ \lambda_i^k \, \mathbb E\, \{ N \langle \boldsymbol v_i ,
\boldsymbol u \rangle ^2 | \lambda_i \} \}
\end{eqnarray}
We assume that the inner expectation is a smooth function
$p(\lambda_i;\theta)$ with a Taylor series expansion $\sum_\ell
\lambda_i^\ell p_\ell(\theta)$. Therefore, we reach to
\begin{equation}
\label{Epl} \frac{1}{N} \sum_{i,\ell} \mathbb E \,
\lambda_i^{k+\ell} p_\ell(\theta) = \sum_\ell p_\ell(\theta)
\frac{1}{N} \sum_{i=1}^N \mathbb E \, \lambda_i^{k+\ell}
\end{equation}
It is known that whatever $\theta$ is, the limiting behavior of
the eigenvalues of $\boldsymbol B$ obeys the semi-circle law
\cite{Feral}, since the deformation is of finite rank and energy.
Therefore, the right-hand-side in \eqref{Epl} converges in
probability to
\begin{equation}
\sum_\ell p_\ell(\theta) \int \mu_{\text {sc}}(x) \,  x^{k+\ell}
\textnormal dx
\end{equation}
and we will have
\begin{equation}
\mathbb E \, \boldsymbol u^* \boldsymbol B^k \boldsymbol u = \int
\mu_{\text {sc}}(x) \, x^k p(x;\theta) \, \textnormal d x
\end{equation}
\end{proof}

To show convergence to the mean value, it will be sufficient to
show the same for the second moment of the quadratic form:

\begin{lemma}
The second moment of the quadratic form converges in probability
to square of its mean value:
\begin{equation}
\label{Eubu2} \mathbb E \, ( \boldsymbol u^* \boldsymbol B^k
\boldsymbol u )^2 \xrightarrow {\quad \textnormal p \quad } ( \,
\mathbb E \, \boldsymbol u^* \boldsymbol B^k \boldsymbol u )^2
\end{equation}
\end{lemma}
\begin{proof}
expanding the terms of quadratic form we have:
\begin{equation}
\label{ubu2} ( \boldsymbol u^* \boldsymbol B^k \boldsymbol u )^2 =
\boldsymbol u^* (\boldsymbol A + \theta \boldsymbol {uu}^*)^k
\boldsymbol u \boldsymbol u^* (\boldsymbol A + \theta \boldsymbol
{uu}^*)^k \boldsymbol u
\end{equation}
Using $\boldsymbol u^* \boldsymbol u = 1$, \eqref{ubu2}  reduces
to a summation of product forms $\theta ^{k_0} ( \boldsymbol u^*
\boldsymbol A^{k_1} \boldsymbol u ) \cdots ( \boldsymbol u^*
\boldsymbol A^{k_\ell} \boldsymbol u )$. In Lemma \ref{lem:uau} we
have shown that each individual term converges in probability to
its mean value. Therefore, their product will also converge in
probability to the product of the mean values. The same is true
for the R.H.S. of \eqref{Eubu2} and therefore the L.H.S. converges
to the R.H.S. in probability.
\end{proof}

\section{Proof of Lemma \ref{lem:chebyshev}}
\label{app:cheb} We first derive $f_0(x)$ and $f_1(x)$ and then
show that $f_k(x)$ obey a recurrence equation which is
characteristic of the Chebyshev polynomials of second kind
\cite{Rivlin}.

\subsection{Calculating $f_0(x)$}
The inner product $\boldsymbol u^* \boldsymbol B^{2k} \boldsymbol
u = \boldsymbol u^* (\boldsymbol A + \theta \boldsymbol
{uu}^*)^{2k} \boldsymbol u$ admits a polynomial expansion in which
the $\theta^0$ term is equal to
\begin{equation}
\label{t0poly} \boldsymbol u^* \boldsymbol A^{2k} \boldsymbol u
\xrightarrow {\quad \textnormal p \quad } c_k \end{equation}
by Lemma \ref{lem:uau}. $\boldsymbol u^* \boldsymbol B^{2k}
\boldsymbol u$ also converges to an integral form which is stated
in Lemma \ref{lem:ubu_int} and gives rise to $\theta^0$ term of:
\begin{equation}
\label{t0} \int \mu_{\text {sc}}(x) \, x^{2k} f_0(x) \textnormal d
x.
\end{equation}
$2k^{\textnormal {th}}$ moments of the semicircle law is known to
be the Catalan number $c_k$ while the odd moments are zero.
Therefore, a Taylor expansion on $f_0(x)$:
\begin{equation}
\label{taylor-f} f_k(x) = \sum_{i=0}^\infty a_{ki} x^i
\end{equation}
applied in \eqref{t0} and equating to \eqref{t0poly} gives:
\begin{equation}
\label{f0even} a_{00}c_k + a_{02} c_{k+1} + a_{04} c_{k+2} +
\cdots \equiv c_k    \quad  \forall k \in \mathbb N
\end{equation}
and therefore, $a_{00}=1$ while $a_{0 (2k)}=0$ for all $k
\geqslant 1$. In the same manner $\boldsymbol u^* \boldsymbol
B^{2k+1} \boldsymbol u$ gives rise to $\theta^0$ term of
$\boldsymbol u^* \boldsymbol A^{2k+1} \boldsymbol u \xrightarrow {
\; \textnormal p \;} 0$ which should be equal to the integral form
of the Taylor series:
\begin{equation}
\label{f0odd} a_{01}c_{k+1} + a_{03} c_{k+2} + a_{05} c_{k+3} +
\cdots \equiv 0    \quad  \forall k \in \mathbb N
\end{equation}
which leads to $a_{0(2k+1)} = 0$ and finally we will have:
\begin{equation}
\label{f0} f_0(x) = 1 = U_0 \left( \frac{x}{2} \right)
\end{equation}
in which, $ U_0(x)$ is the zero$^{\textnormal {th}}$ Chebyshev
polynomial of second kind.

\subsection{Calculating $f_1(x)$}
Calculating $f_1(x)$ amounts to the $\theta^1$ term of
$\boldsymbol u^* (\boldsymbol A + \theta \boldsymbol
{uu}^*)^{2k+1} \boldsymbol u$. The general term of this binomial
expansion with only one $\theta$ term is $(\boldsymbol u^*
\boldsymbol A^n \theta \boldsymbol u ) (\boldsymbol u^*
\boldsymbol A^{2k-n} \boldsymbol u ) \xrightarrow {\quad
\textnormal p \quad } c_{\frac{n}{2}} c_{k-\frac{n}{2}}$ when
$n=2m$ is an even number  $0 \leqslant n \leqslant 2k$ according
to Lemma \ref{lem:uau}. Therefore, sum of these general terms will
give the $\theta^1$ coefficient as:
\begin{equation}
\label{f1sum} \sum_{m=0}^k c_m c_{k-m} = c_{k+1}
\end{equation}
by the well-known recurrence of Catalan numbers \cite{Anderson}.
The integral form on the Taylor series expansion of $f_1(x)$ gives
the following equivalence:
\begin{equation}
\label{f1sum} a_{11} c_{k+1} + a_{13} c_{k+2} + a_{15} c_{k+3} +
\cdots \equiv c_{k+1}  \quad  \forall k \in \mathbb N
\end{equation}
which leads to $a_{11}=1$ and $a_{1(2k+1)}=0$ for all $k \geqslant
1$.

To determine even coefficients we consider $\boldsymbol u^*
(\boldsymbol A + \theta \boldsymbol {uu}^*)^{2k} \boldsymbol u$.
From $2k$ selections between $\boldsymbol A$ and $\theta
(\boldsymbol {uu}^*)$, one of them is $\theta$ and the general
term of interest is $(\boldsymbol u^* \boldsymbol A^n \theta
\boldsymbol u ) (\boldsymbol u^* \boldsymbol A^{2k-1-n}
\boldsymbol u )$ in which $0 \leqslant n \leqslant 2k-1$. Here it
is impossible that both $n$ and $2k-1-n$ be even numbers.
Therefore, the equivalence of the integral form and the
combinatorial term will be:
\begin{equation}
\label{f1even} a_{10}c_k + a_{12} c_{k+1} + a_{14} c_{k+2} +
\cdots \equiv 0    \quad  \forall k \in \mathbb N
\end{equation}
which gives $a_{1(2k)}=0$. Therefore, we will have:
\begin{equation}
\label{f1} f_1(x) = x = U_1 \left( \frac{x}{2} \right)
\end{equation}

\subsection{Recurrence of $f_n(x)$}
Chebyshev polynomials of second kind satisfy the following
recurrence equation:
\begin{equation}
\label{urec} U_n(x) = 2 \, x U_{n-1}(x) - U_{n-2}(x)
\end{equation}
Therefore, we should prove a similar recurrence on $f_n(x)$:
\begin{equation}
\label{frec} f_n(x) = x f_{n-1}(x) - f_{n-2}(x)
\end{equation}

Define $H(m,n)$ as the coefficient of $\theta^n$ in the inner
product $\boldsymbol u^* (\boldsymbol A + \theta \boldsymbol
{uu}^*)^m \boldsymbol u$. We can show a recurrence on $H(m,n)$:
\begin{equation}
\label{hrec} H(m,n) = H(m+1,n-1) - H(m,n-2)
\end{equation}
To show \eqref{hrec}, assume a sequence of elements $A$ and
$\theta$ with length $m$. We are interested in the sum of all
sequences with predetermined number of $\theta$'s while runs of
$A$ should be even and each run of $A^{2n}$ is translated to
$c_n$. For example:
\begin{equation}
\label{ex} AA \theta AAAA \theta AAAAAA \longrightarrow c_1 c_2
c_3
\end{equation}
Therefore, sum of all sequences with $n$ elements of $\theta$ and
total length $m$ is $H(m,n)$. To further translate the problem to
a combinatorial object enumeration, we use the fact that the
Catalan number $c_k$ counts the number of Dyck paths with length
$2k$. Dyck paths are bernouli $\pm 1$ random walks which are
always above the horizontal zero level and starting and ending in
zero level. Therefore, transformation \eqref{ex} amounts to
counting the number of paths with a Dyck path of length 2, then a
horizontal zero level (h) step forward, then a Dyck path of length
4 ($D_4$), h, and finally a $D_6$. Total number of such paths with
$n$ h-steps and total length $m$ is $H(m,n)$. Dyck paths are bound
to even length. For ease of notation, define $I_{mn}$ as the
number of paths with total length of Dyck paths $2m$ and number of
h-steps $n$. Note that h-steps can only occur in the zero level.
Fig. \ref{fig:dyck} shows an example of such a path.

\begin{figure}
\centering
\includegraphics[width = 0.45 \textwidth]{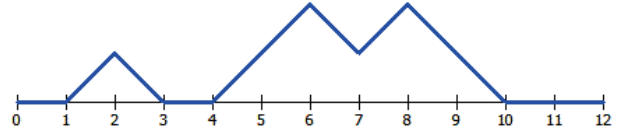}
\caption{A sample path with total length 12 from which 4 steps are
h-steps while 8 steps are Dyck $\pm 1$ steps. h-steps only occur
in the zero level. Number of all such paths with 4 h-steps and $2
\times 4$ Dyck steps are denoted as $I_{4,4}=H(12,4)$.}
\label{fig:dyck}
\end{figure}

Dyck paths are equivalent to planar trees. In fact each Dyck path
determines a unique planar tree by the following construction:
Start from zero level and add a root node. With each up step add a
new edge and the corresponding node to the tree and move up the
tree to the new node. With each down step move down one node. For
the paths with h-steps define a second type of dashed edges and
add a dashed edge and a new node each time a h-step occurred.
Therefore, our paths correspond to some planar trees connected in
order by dashed edges from root nodes. An example of such trees is
depicted in Fig. \ref{fig:tree} which is equivalent to the path in
Fig. \ref{fig:dyck}.

\begin{figure}
\centering
\includegraphics[width = 0.2 \textwidth]{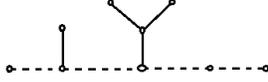}
\caption{The tree equivalent to the path in Fig. \ref{fig:dyck}.
The path starts from left, goes up when reaching a solid tree,
come down the tree and then again to the right.} \label{fig:tree}
\end{figure}

We use analytic combinatorics \cite{Flajolet} to show that
$I_{mn}$ satisfies the following recurrence:
\begin{equation}
\label{Irec} I_{mn} = I_{(m+1)(n-1)} - I_{(m+1)(n-2)}
\end{equation}
for $m \geqslant 0$ and $n\geqslant 2$. Associate variable $x$ to
each solid edge in the compound tree and variable $y$ to each
dashed edge. A planar tree is a node with a sequence of trees
attached to it. In fact trees are recursive combinatorial objects.
As an example consider sequences of object $x$ regardless of their
length:
\begin{equation}
\label{seq} {\mathcal S \textnormal {EQ}}(x) \mathrel{\mathop :} =
\{ \phi , x , xx , xxx , \cdots \}
\end{equation}
in which $\phi$ is the null sequence. These sequences correspond
to a generating function in which power of $x$ determines number
of $x$'s in the sequence:
\begin{equation}
\label{seqg} \textsc {Seq}(x) \mathrel{\mathop :} = \frac{1}{1-x}
= 1+x+x^2+x^3+\cdots
\end{equation}
In the same manner, sequences of $x$ and $y$'s are demonstrated
using the following generating function:
\begin{equation}
\label{seqmix} \textsc {Seq}(x,y) \mathrel{\mathop :} =
\frac{1}{1-x-y} = \sum_{n=0}^\infty (x+y)^n
\end{equation}
since each sequence with length $n$ corresponds to a binomial
expansion of $(x+y)^n$.

Planar trees are sequences of planar trees attached to a single
root node:
\begin{equation}
\label{tree} \mathcal T(x) \mathrel{\mathop :} = \epsilon \;
\mathcal S \textnormal {EQ} (\mathcal T(x))
\end{equation}
in which $\epsilon$ is the root node. Therefore, the generating
function of a planar tree is:
\begin{equation}
\label{treeg} \textsc{T} (x) = \frac{1}{1- \textsc{T}(x)}
\end{equation}
which gives:
\begin{equation}
T(x) = \frac{1}{2} \left( 1-\sqrt{1-4x} \right)
\end{equation}

Our paths are sequences of solid planar trees and dashed edges:
\begin{equation}
\label{w} \textnormal W(x,y) = \frac{1}{1-y-\frac{1}{2} \left(
1-\sqrt{1-4x} \right)}
\end{equation}
Number of paths with $m$ solid edges and $n$ dashed edges is the
coefficient of $x^m y^n$ in Taylor series expansion of $W(x,y)$
and therefore:
\begin{equation}
W(x,y) = \sum_{m,n=0}^\infty I_{mn} x^m y^n
\end{equation}
The recurrence in \eqref{Irec} can be shown on the explicit form
of $W(x,y)$ in \eqref{w}. The problem is that \eqref{Irec} is only
valid for $n \geqslant 2$. If we define number of paths with $-1$
dashed paths zero $I_{m(-1)}=0$, then \eqref{Irec} will be valid
for $n \geqslant 1$. Therefore, \eqref{Irec} is equivalent to:
\begin{equation}
\label{wrec} W - W_{:0} = \frac{y}{x} (W - W_{0:}) - \frac{y^2}{x}
(W - W_{0:})
\end{equation}
in which $W_{:0}= W(x,y=0)$ is the coefficient of $y^0$ in Taylor
series of $W(x,y)$ and $W_{0:} = W(x=0,y)$ is the coefficient of
$x^0$. \eqref{wrec} is easily verifiable. This gives \eqref{Irec}
and then \eqref{hrec}. $H(m,n)$ was defined in \eqref{hrec} as the
sum of coefficients of $\theta^n$ in $\boldsymbol u^* \boldsymbol
B^m \boldsymbol u$ and therefore is equal to the integral form in
Lemma \ref{lem:ubu_int} and we will have:
\begin{equation}
\int x^m f_n \textnormal d\mu_{\textnormal {sc}} = \int  x^{m+1}
f_{n-1} \textnormal d\mu_{\textnormal {sc}} - \int x^m f_{n-2}
\textnormal d\mu_{\textnormal {sc}}
\end{equation}
for all $m \geqslant 0$ and fixed $n \geqslant 2$. This completes
the proof of \eqref{frec} and Lemma \ref{lem:chebyshev}.

\end{document}